\documentclass[11pt]{amsart}
\usepackage{graphicx}
\usepackage{amscd,amsthm}
\usepackage{amsmath,color}
\usepackage{amsfonts}
\usepackage{amssymb,appendix,wasysym}
\usepackage[latin1]{inputenc}
\usepackage[colorlinks=true]{hyperref}
\usepackage{mathrsfs}
\usepackage{yfonts}
\newtheorem{theorem}{Theorem}[section]
\newtheorem{lemma}{Lemma}[section]
\newtheorem{corollary}{Corollary}[section]
\newtheorem{proposition}{Proposition}[section]
\newtheorem{remark}{Remark}[section]
\newtheorem{example}{Example}[section]
\newtheorem{definition}{Definition}[section]

\newcommand{\RT}{\mathbb{R}^2}
\newcommand{\R}{\mathbb R}
\newcommand{\e}{\varepsilon}
\newcommand{\ds}{\displaystyle}

\newcommand{\RN}{\mathbb R^N}

\newcommand{\iy}{\infty}

\newcommand{\dd}{\delta}

\newcommand{\na}{\nabla}

\newcommand{\al}{\alpha}

\newcommand{\ti}{\tilde}

\newcommand{\re}[1]{(\ref{#1})}

\newcommand{\rg}{\rightarrow}

\newcommand{\lab}{\label}
\newcommand{\bt}{\begin{theorem}}
\newcommand{\et}{\end{theorem}}
\newcommand{\bl}{\begin{lemma}}
\newcommand{\el}{\end{lemma}}
\newcommand{\bd}{\begin{definition}}
\newcommand{\ed}{\end{definition}}
\newcommand{\bc}{\begin{corollary}}
\newcommand{\ec}{\end{corollary}}
\newcommand{\bp}{\begin{proof}}
\newcommand{\ep}{\end{proof}}
\newcommand{\bx}{\begin{example}}
\newcommand{\ex}{\end{example}}
\newcommand{\bi}{\begin{exercise}}
\newcommand{\ei}{\end{exercise}}
\newcommand{\bo}{\begin{proposition}}
\newcommand{\eo}{\end{proposition}}
\newcommand{\br}{\begin{remark}}
\newcommand{\er}{\end{remark}}
\newcommand{\be}{\begin{equation}}
\newcommand{\ee}{\end{equation}}
\newcommand{\ba}{\begin{align}}
\newcommand{\ea}{\end{align}}
\newcommand{\bn}{\begin{enumerate}}
\newcommand{\en}{\end{enumerate}}
\newcommand{\bg}{\begin{align*}}
\newcommand{\eg}{\end{align*}}
\newcommand{\bcs}{\begin{cases}}
\newcommand{\ecs}{\end{cases}}

\newcommand{\bean}{\begin{eqnarray*}}
\newcommand{\eean}{\end{eqnarray*}}

\newcommand{\ud}{\mathrm{d}}
\title[Schr\"{o}dinger equation in $\mathbb R^2$ with saddle potential]
{Solutions concentrating around the saddle points of the potential for Schr\"{o}dinger equations with critical exponential growth }

\author[J. Zhang]{Jianjun Zhang}
\author[J. M.\ do \'O]{Jo\~ao Marcos do \'O}
\author[P. K. Mishra]{Pawan Kumar Mishra}

\address[J. Zhang]{\newline\indent School of Science
\newline\indent
Chongqing Jiaotong University
\newline\indent
Chongqing 400074, PR China}
\email{\href{mailto:zhangjianjun09@tsinghua.org.cn}{zhangjianjun09@tsinghua.org.cn}}

\address[J. M. do \'O]{\newline\indent Department of Mathematics
\newline\indent
Federal University of Para\'{\i}ba
\newline\indent
58051-900, Jo\~ao Pessoa-PB, Brazil}
\email{\href{mailto:jmbo@pq.cnpq.br}{jmbo@pq.cnpq.br}}

\address[P. K. Mishra]{\newline\indent Department of Mathematics
\newline\indent
Federal University of Para\'{\i}ba
\newline\indent
58051-900, Jo\~ao Pessoa-PB, Brazil}
\email{\href{mailto:pawanmishra31284@gmail.com}{pawanmishra31284@gmail.com}}

\thanks{Corresponding author: J.~M.~do~\'{O}. \\
Research supported in part by INCTmat/MCT/Brazil, CNPq and CAPES/Brazil.}

\begin{document}
\maketitle
\begin{abstract}
In this paper, we deal with the following nonlinear Schr\"odinger equation

$$
-\epsilon^2\Delta u+V(x)u=f(u),\ u\in H^1(\mathbb R^2),
$$
where $f(t)$ has critical growth of Trudinger-Moser type.
By using the variational techniques, we construct a positive solution $u_\epsilon$ concentrating around the saddle points of the potential $V(x)$ as $\epsilon\rightarrow 0$.
Our results complete the analysis made in \cite{MR2900480} and \cite{MR3426106}, where the Schr\"odinger equation was studied in $\mathbb R^N$, $N\geq 3$ for sub-critical and critical case respectively in the sense of Sobolev embedding. Moreover, we relax the  monotonicity condition on the nonlinear term $f(t)/t$ together with a compactness assumption on the potential $V(x)$, imposed in  \cite{MR3503193}.
\end{abstract}


\section{Introduction}

In the last decades, considerable attention has been paid to the standing wave solutions of the following nonlinear Schr\"{o}dinger equation

\begin{equation}\label{P}
-\e^2\Delta v+V(x)v=f(v),\ \ v\in H^1(\R^N), N\geq 3.
\end{equation}
 An interesting class of solutions of
\eqref{P} are families of solutions which develop a spike shape
around some certain point in $\R^N$ as $\e\rightarrow 0$.
 Without giving an exhaustive list of references, we cite
\cite{MR2317788,MR2168823,MR2574884,MR1379196,MR1736974,MR2334939,MR1342381,MR1219814,MR1218300}.

\vskip0.1in

Recall that based on a Lyapunov-Schmidt reduction, for $N=1$ and $f(t)=t^3$ A.~Floer and A.~Weinstein \cite{MR867665} first  constructed a single peak solution of \eqref{P} which concentrates around any
given non-degenerate critical point of $V(x)$. For
$f(t)=|t|^{p-2}t, \; p\in (2,2^{\ast})$, Y.~G.~Oh \cite{MR970154} extended the result in \cite{MR867665} to the higher dimension case. In \cite{MR867665,MR970154}, their arguments requires essentially a non-degeneracy condition.

A first attempt to generalize the results of \cite{MR867665,MR970154}, without non degeneracy condition,
 was made in \cite{MR1162728} see also \cite{MR1391524,MR1379196}.
 In
\cite{MR1162728}, by using purely variational approach,
P.~H.~Rabinowitz proved the existence of positive solutions of \eqref{P}
for small $\e>0$ whenever
$$
\liminf_{|x|\rightarrow\infty}V(x)>\inf_{\RN}V(x).
$$
Subsequently these results were
improved by M. del Pino and P. Felmer in \cite{MR1471107,MR1931757} using a variational approach applied to a truncated
problem and assuming Ambrosetti-Rabinowitz((A-R) for short) and monotonicity condition on the nonlinearity.
 It is natural to ask whether this result holds for more general nonlinear term $f(t)$,
particularly, without $(A$-$R)$ or the monotonicity condition on $f(t)/t$. In 2007, in \cite{MR2317788}, J.~Byeon and J.~Jeanjean responded with an affirmative answer for $N\ge3$. Precisely,
with the Berestycki-Lions conditions (cf. (1.1)-(1.3) in \cite{MR695535} or $(f1)$-$(f3)$ in \cite{MR2317788}),  authors in \cite{MR2317788} developed a new variational method to construct
the spike solutions of \eqref{P}, which concentrate around the local minimum of $V$.  In these references only the case of the local minima of $V(x)$ was considered.
\subsection{Assumptions and Related results}
Recently in \cite{MR2900480}, P.~d'Avenia et. al. have studied the existence of spike solutions around saddle or maximum points of $V(x)$ without
assuming the monotonicity on the nonlinearity $f(t)/t$ and with the potential satisfying the following assumptions
\begin{enumerate}

\item [{(V0)}] $0<\alpha_1 \leq V(x) \leq \alpha_2$, for all $x \in \R^N$;

\end{enumerate}
Moreover, with respect to the critical point $0$, it was assumed that the potential $V(x)$ satisfies
one of the following conditions:

\begin{enumerate}

\item [({V1})] $V(0)=1$, $V$ is $C^1$ in a neighborhood of $0$
and $0$ is an isolated local maximum point of $V$.

\item [({V2})] $V(0)=1$, $V$ is $C^2$ in a neighborhood of $0$
and $0$ is a non-degenerate saddle critical point of $V$.

\item [({V3})] $V(0)=1$, $V$ is $C^{N-1}$ in a neighborhood of
$0$, $0$ is an isolated critical point of $V(x)$ and there exists
a vector space $E$ such that:

\begin{enumerate}

\item $V|_E$ has a local maximum at $0$;

\item $V|_{E^\perp}$ has a local minimum at $0$.

\end{enumerate}

\end{enumerate}
In \cite{MR2900480}, for sub-critical nonlinearity, following the
approach as in \cite{MR1931757}, the authors have defined a modified energy functional and
constructed a different mini-max argument which involves suitable deformations of certain cones in
$H^1(\R^N)$. Further this result for critical nonlinearity was studied in \cite{MR3426106}.
Before stating our main result, we shall introduce the main hypotheses on $f(t)$. In what follows, we assume that $f\in
C(\R,\R)$ and satisfies
\begin{itemize}
\item[$({f_1})$] $\lim_{t\rightarrow 0}{f(t)}/{t}=0$.
\item[$({f_2})$] $\lim_{s\rightarrow+\infty}\frac{f(s)}{e^{\alpha s^2}-1}=\left\{
\begin{array}{ll}
0,\ \ \ \ \ \ \forall \alpha>4\pi,\\
+\infty,\ \ \ \forall \alpha<4\pi.
\end{array}
\right.
$
\item[$({f_3})$]   There exists $\mu>2$ such that $tf(t)\geq \mu F(t)>0$, where  $F(t):=\int_0^tf(\tau) \, \mathrm{d} \tau$.
    \item [$({f_4})$] There exist  $p>2$, $\lambda_p>0$ such that $f(t)\ge\lambda_p t^{p-1}$ for all $t> 0$.

   \end{itemize}

\subsection{Motivation and Main result}
In a very interesting work \cite{MR3082247}, under the Berestycki-Lions conditions, J.~Byeon and K.~Tanaka improved the result in \cite{MR2317788} and obtained the existence of positive solutions to \eqref{P}, which concentrate around more general critical points (such as local maximum points and special saddle points) of $V(x)$. See also \cite{MR2900480} for related results. Further, J.~Zhang, Z.~Chen and W.~Zou \cite{MR3291802} extended the result in \cite{MR2317788} to the critical case and general nonlinear term $f(t)$.

In 2008, J. Byeon, L. Jeanjean and K. Tanaka \cite{MR2424391} considered the
concentration phenomenon of the problem  around the local minima of the potential in the cases: $N=1,2$.
In particular, for $N=2$ they assume that $f\in C(\R^+, \R^+)$ and satisfies the subcritical growth.
In \cite{MR1846738}, J. M. do \'{O} and M. A. S. Souto proved the existence of one spike solution around a local minima of $V(x)$, where the nonlinear term $f(t)$ has critical growth of Trudinger-Moser type at $\infty$, i.e., $f(t)$ behaves like $\exp(\alpha_0 t^2)$ for some $\alpha_0>0$ as $t\rightarrow+\infty$. In \cite{MR1846738} $(A$-$R)$, the monotonicity and other conditions on $f(t)$ were also required.\\

Motivated by \cite{MR1846738, MR2900480}, we consider the two-dimensional case of problem \eqref{P}. To be more precise, we study the concentration phenomenon of the following problem around saddle points of the potential $V(x)$,
\begin{equation}\label{PE}
-\e^2\Delta v+V(x)v=f(v),\ \ v\in H^1(\R^2),
\tag{$P_\varepsilon $}
\end{equation}
where $f(t)$ has the maximal growth on $t$ which allows us to treat this problem
variationally in $H^1(\R^2)$ motivated by the Trudinger--Moser type inequality due to D.~Cao \cite{MR1163431} (see also \cite{MR1704875}). This result can be compared with \cite{MR2900480, MR1846738} as follows. In \cite{MR2900480} the problem was studied with subcritical nonlinearity of Sobolev type and the concentration behavior around the saddle points of the potential $V(x)$. On the other hand, in \cite{MR1846738} the problem was considered with the maximal growth on nonlinearity of Trudinger-Moser type, but the concentration behavior was investigated around the local minima of the potential.

Let $\mathcal{C}_p$ be denoted by the best constant in the Sobolev embedding of $H^1(\mathbb{R}^2)$ into $ L^{p}(\mathbb{R}^2)$, i.e.,
$$
\mathcal{C}_p\left(\int_{\mathbb{R}^2}|u|^p\,\ud x \right)^{{2}/{p}}\le\int_{\mathbb{R}^2}\left(|\nabla u|^2+u^2\right)\ud x,\quad \forall u\in H^1(\mathbb{R}^2).
$$
The main theorem of this paper reads as
\begin{theorem} \label{thm1} Assume that $f$ satisfies hypotheses $(f1)$-$(f4)$ with
$$\lambda_p>\left(\frac{p-2}{p}\cdot\frac{2\mu}{\mu-2}\right)^{\frac{p-2}{2}}\mathcal{C}_p^{\frac{p}{2}},$$
and  $V(x)$ satisfies $(V0)$ and one of $(V1)$, $(V2)$ or $(V3)$. Then there exists $\e_0>0$ such that
 \eqref{PE} admits a positive solution $u_{\e}$ for $\e \in (0,
\e_0)$. Moreover, there exists $\{y_{\e}\} \subset \R^2$ such that
$\e y_{\e} \to 0$ and $u_{\e}(\e ( \cdot + y_{\e}))  $ converges to a ground state solution of
\begin{equation}\lab{wolin}
 - \Delta u + u = f(u),\;\; u>0,\;\; u\in H^1(\R^2).
\end{equation}
\end{theorem}

We mention that the nonlinear elliptic problems involving critical growth of Trudinger-Moser type have been studied by many authors; see, for example, Adimurthi \cite{MR1079983}, D.~Cao \cite{MR1163431}, de Figueiredo et al. \cite{MR2772124}, J.~M.~do~\'{O} et al. \cite{MR1846738, MR1704875} and N.~Lam, G.~Lu \cite{MR3145918}.

The rest of the paper is organized as follows. In Section 2, we give some preliminary results and the variational setting. In Section 3, we define a truncation of the problem which will be used throughout the present paper. Moreover, a compactness result is established. In Section 4, we introduce some results related to the corresponding limit problem, most of which are well-known. The min-max argument is exposed in Section 5, where we show the main estimates of the min-max level $m_\e$ associated to the truncation problem. These estimates play a crucial role in proving the existence of solutions. In Section 6, a key asymptotic estimate on $m_\e$ is proved. Finally, in Section 7, we show that the solutions of the truncated problem actually solve the original problem for $\e$ small.

\section{Preliminaries and variational setup}
In this section we give some preliminary definitions and
results which will be used in our subsequent arguments. We will follow the following notations. For any $R>0$, $B(x,R)$ denotes
the ball centered at $x$ and with radius $R>0$ in $\mathbb R^2$. Moreover, for any $\e>0$, let
$$ \Lambda^\e := \e^{-1} \Lambda = \left\{ x\in\RT\;\vline\; \e x \in
\Lambda \right\}. $$
Using the change of variable $x \mapsto \e x$,
problem \eqref{PE} is transformed as
\begin{equation}
\label{trunp} - \Delta u + V(\e x) u = f(u) \qquad \hbox{in } \RT.
\end{equation}
The energy functional $I_{\e}:H^1(\RT) \to \R$, corresponding to problem \eqref{trunp} is given by
\[
I_\e (u) = \frac{1}{2}\int_{\RT} (|\nabla u|^2+V(\e x)u^2)\,\ud x - \int_{\RT} F(u)\,\ud x
\]
 and defined on the Hilbert space $H^1(\mathbb R^2)$ with the inner product and norm given by

 \[
 \langle u,v \rangle =\displaystyle \int_{\mathbb R^2}\left(\nabla u\nabla v+uv\right)\,\ud x,\,\,\,\|u\|^2=\langle u,v \rangle=\displaystyle \int_{\mathbb R^2}\left(|\nabla u|^2+u^2\right)\,\ud x.
 \]
 Using the following Moser-Trudinger inequality (see \cite{MR1704875}), it is standard to verify that the functional $I_\e$ is well defined and $C^1$ with the Fr\'echet derivative given by
 \[
 \langle I^\prime_\e (u), \phi\rangle  = \int_{\RT} (\nabla u\nabla \phi + V(\e x)u\phi)\,\ud x - \int_{\RT} f(u)\phi\,\ud x,\,\mbox{for}\,\,\phi\in C_0^\iy(\R^2).
 \]
\begin{lemma}\label{MTI}
If $u\in H^1(\mathbb R^2)$ and $\alpha>0$, then
\[
\displaystyle \int_{\mathbb R^2}\left(e^{\alpha u^2}-1\right)\,\ud x<\infty.
\]
Moreover, if $\|\nabla u\|_{L^2}\leq 1$, $\|u\|_{L^2}\leq M$ and $\alpha<4\pi$, then there exists a constant $C$ which depends only on $\alpha$ and $M$, such that
\[
\displaystyle \int_{\mathbb R^2}\left(e^{\alpha u^2}-1\right)\,\ud x\le C.
\]
\end{lemma}
Throughout this paper, we denote standard norm in  $H^1(\RT)$ by $\| \cdot \|$. The strong and weak convergence of sequences of functions are denoted by $\rightarrow$ and $\rightharpoonup$ respectively in the space $H^1(\RT)$.

\section{Compactness result for truncated problem}
We will not study problem \eqref{trunp} in its original form. First, we will make a suitable truncation of the
nonlinearity $f(t)$. Then we will find a solution of the truncated
problem. At the end, under a suitable control over the solution of the truncated problem, we come back to the solution of the original problem. For that, let us define
\[
\tilde{f}(t)= \left\{
\begin{array}{lll}
\min\{f(t),at\}, & & t\geq 0,\\
0, & &t<0,
\end{array}
\right.
\]
with
\begin{equation*}
0<a<\left(1-\frac{2}{\mu}\right)\al_1,
\end{equation*}
where $\alpha_1$ is introduced in $(V0)$.
In the following we consider the balls $B_i:=B(0,R_i)\subset\RT$ ($i=0,1,2,\ldots,5$) with $R_i <R_{i+1}$ for
$i=0,1,2,3, 4$, where $R_i$ are small positive constants and will be determined later. For some technical reasons, we choose $R_1$ satisfying
\begin{equation*}  \forall x \in \partial
B_1 \mbox{ with } V(x)=1, \partial_{\tau}V(x) \neq 0, \mbox{ where
} \tau \mbox{ is tangent to } \partial B_1 \mbox{ at }x.
\end{equation*}

Next we define $\chi : \R^2 \to \R$,
\begin{equation*}
\chi(x)= \left \{
\begin{array}{ll}
 1, & x \in B_1, \\
 \frac{R_2 - |x|}{R_2-R_1}, & x \in B_2 \setminus B_1, \\
 0, & x \in B_2^c
\end{array}
\right.
\end{equation*}
and then
\begin{align*}
g(x,t)=&\chi(x) f(t) + \left( 1- \chi(x)\right) \tilde{f}(t),\\
G (x,t):=&\int_0^t g(x,\tau)\,\ud\tau = \chi(x) F(t) + \left( 1-
\chi(x)\right) \tilde{F}(t),
\end{align*}
where $F(t)$ and $\tilde {F}(t)$ are primitives of $f(t)$ and $\tilde{f}(t)$ respectively.
We also denote
\[\chi_\e(x)=\chi(\e x) \] and
\[
g_\e (x,t):= g(\e x, t)= \chi_\e(x) f(t) + \left( 1-
\chi_\e(x)\right) \tilde{f}(t).
\]
Therefore, the truncated problem related to problem \eqref{trunp} looks like the following
\begin{equation}\label{truneq}
 - \Delta u + V(\e x) u = g_{\e}(x,u) \qquad
\hbox{in } \RT.
\tag{$\tilde {P}_\varepsilon $}
\end{equation}
The solutions of \eqref{truneq} are the  critical points of the associated energy functional
$\tilde{I}_{\e}:H^1(\RT) \to \R$, defined as
\[
\tilde{I}_\e (u) = \frac{1}{2}\int_{\RT} (|\nabla u|^2 + V(\e x) u^2)\,\ud x - \int_{\RT} G_\e(x,u)\,\ud x,
\]
where
\[
G_\e (x,t):=\int_0^t g_\e (x,\tau)\,\ud \tau = \chi_\e(x) F(t) +
\left( 1- \chi_\e(x)\right) \tilde{F}(t).
\]
We now state the following compactness result.

\begin{proposition}
\label{PSC} For every $\e>0$, the functional $\tilde{I}_\e$
satisfies the Palais-Smale condition , $(PS)_c$ in short, for all $c<\frac{1}{2}\left(\frac{1}{2}-\frac{1}{\mu}\right)$, where $\mu$ is introduced in $(f3)$.
\end{proposition}

\begin{proof}
\textbf{Step 1:} The sequence $\{u_n\}$ is bounded.\\
Since $\{u_n\}$ is a $(PS)_c$ sequence for $\tilde{I}_\e$,
\[
\tilde{I}_\e (u_n)=\frac{1}{2}\int_{\RT} (|\nabla u_n|^2+V(\e x) u_n^2)\,\ud x - \int_{\RT} G_\e(x,u_n)\,\ud x \to c
\]
and
\[
\langle \tilde{I}'_\e (u_n),u_n \rangle=\int_{\RT} (|\nabla u_n|^2 + V(\e x) u_n^2)\,\ud x
- \int_{\RT} g_\e(x,u_n) u_n\,\ud x = o(\|u_n\|).
\]
Then, by $(f3)$ we have
\begin{align*}
&\mu \tilde{I}_\e(u_n) - \langle \tilde{I}'_\e (u_n),u_n \rangle\\
& =\left( \frac{\mu}{2}-1\right) \int_{\RT} \left ( |\nabla u_n|^2  + V(\e x) u_n^2 \right )\,\ud x - \int_{\RT} \chi_\e (x) \left(\mu F(u_n) - f(u_n) u_n\right)\,\ud x\\
&\,\,\,\,\,\,\, - \int_{\RT} \left(1- \chi_\e (x)\right) \left(\mu \tilde{F}(u_n) - \tilde{f}(u_n)u_n\right)\,\ud x\\
& \geq \left( \frac{\mu}{2}-1\right) \int_{\RT} \left ( |\nabla u_n|^2  + V(\e x) u_n^2 \right )\,\ud x - \frac{\mu}{2}a  \int_{\RT\setminus B_1^\e} u_n^2\,\ud x\\
&\geqq \left[\left(\frac{\mu}{2}-1\right)\alpha_1-\frac{\mu}{2}a\right] \| u_n\|^2
\end{align*}
which implies that $\{u_n\}$ is bounded in $H^1(\R^2)$.
\vskip0.1in
\textbf{Step 2:} For any given $\delta>0$, there exists $R=R(\delta)>0$ such that
\be\lab{s}
\limsup_{n\rg\iy} \|u_n\|_{H^1(B(0,R)^c)}\le\delta.
\ee
As a sequence, there exists $\e_0>0$ such that for any $\e\in(0,\e_0)$, we have
$$
\limsup_{n\rg\iy}\|\na u_n\|_2^2<\frac{2\mu c_0}{\mu-2}<\frac{1}{2}
$$
for some $c_0\in\left(c,\frac{1}{2}\left(\frac{1}{2}-\frac{1}{\mu}\right)\right)$. In fact, by Step 1, there exists $\e_0>0$ such that for $\e\in(0,\e_0)$, $\R^2\setminus B_1^\e\subset B(0,R)$. Then $\limsup_{n\rg\iy}\int_{\RT\setminus B_1^\e} u_n^2\,\ud x\le\frac{\delta^2}{\al_1}$,
\begin{align*}
\limsup_{n\rg\iy}\int_{\RT} |\nabla u_n|^2\,\ud x&\le\limsup_{n\rg\iy}\frac{2}{\mu-2}\left[\mu \tilde{I}_\e(u_n) - \langle \tilde{I}'_\e (u_n),u_n\rangle +\frac{\mu}{2}a  \int_{\RT\setminus B_1^\e} u_n^2\,\ud x\right]\\
&\le\frac{2}{\mu-2}\left[\mu c +\frac{\mu a\delta^2}{2\al_1}\right].
\end{align*}
It follows that by choosing $\delta>0$ small, we have $
\limsup_{n\rg\iy}\|\na u_n\|_2^2<\frac{2\mu c_0}{\mu-2}<\frac{1}{2},
$
for some $c_0\in\left(c,\frac{1}{2}\left(\frac{1}{2}-\frac{1}{\mu}\right)\right)$.

In the following, we prove \re{s}. We take $R>0$ such that $B_2^\e \subset B(0,R/2)$. Let $\phi_R$
a cut-off function such that $\phi_R=0$ in $B(0,R/2)$, $\phi_R=1$
in $B(0,R)^c$, $0\le\phi_R\leq 1$ and $|\nabla \phi_R|\leq C/R$. Then
\[
\langle \tilde{I}'_\e (u_n), \phi_R u_n\rangle =\int_{\RT} [\nabla u_n \cdot \nabla (\phi_R
u_n)+ V(\e x) u_n^2 \phi_R]\,\ud x  - \int_{\RT} g_\e(x,u_n) \phi_R
u_n \,\ud x= o_n(1).
\]
Since $\{u_n\}$ is bounded, therefore
\begin{align*}
\int_{\RT} \left(|\nabla u_n|^2 + V(\e x) u_n^2\right) \phi_R\,\ud x
= &  \int_{\RT} \tilde{f}(u_n) \phi_R u_n\,\ud x - \int_{\RT} u_n \nabla u_n \cdot \nabla \phi_R\,\ud x + o_n(1)\\
\leq & \ a \int_{\RT}  u_n^2\,\ud x +\frac CR + o_n(1)
\end{align*}
which implies
$$
\| u_n \|^2_{H^1(B(0,R)^c)} \leq C/R +o_n(1).$$
Taking $R$ large, the proof of Step 2 follows.
\vskip0.1in

\textbf{Step 3:} For any $\delta$ given above and small enough, we claim that
\[
\displaystyle\limsup_{n\rightarrow\infty}\left(\int_{\mathbb R^2\setminus  B_R(0)}g_\epsilon(x, u_n)u_n\,\ud x+\int_{\mathbb R^2\setminus B_R(0)}g_\epsilon(x, u)u\,\ud x\right)\le\delta.
\]
Indeed, by $(f_1)$-$(f_2)$, for any $\e>0, \alpha>4\pi,q>1$, there exists some $C_1>0$ (independent of $\e,n$) such that
$$
|g_\e(x,t)|\le C_1|t|+|t|(e^{\alpha t^2}-1),\,\, t\in\R.
$$
Then
\begin{align*}
\int_{\mathbb R^2\setminus B_R(0)}g_\epsilon(x, u_n)u_n\,\ud x &\leq C_1 \int_{\mathbb R^2\setminus  B_R(0)} u_n^2\,\ud x+ \int_{\mathbb R^2\setminus  B_R(0)} u_n^2 (e^{\alpha u_n^2}-1)\,\ud x\\
&\leq C_1\|u_n\|^2_{H^1(B^c_R(0))}+\left(\int_{\mathbb R^2\setminus  B_R(0)}|u_n|^{2q'}\,\ud x\right)^\frac{1}{q'}\left(\int_{\mathbb R^2}(e^{q\alpha u_n^2}-1)\,\ud x\right)^\frac1q,
\end{align*}
where $q,q'>1$ and $1/q+1/q'=1$. Choosing $q>0$ (close to 1) and $\al>4\pi$ (close to $4\pi$) such that $q\alpha\|\na u_n\|_2^2<4\pi$ for $n$ large, by Step 2 and  Lemma \ref{MTI}, there exists $C>0$ (independent of $n, R$) such that for $n$ large,
$$
\int_{\mathbb R^2\setminus  B_R(0)}g_\epsilon(x, u_n)u_n\,\ud x\leq C_1 \delta^2+ C\|u_n\|^2_{L^{2q'}(B^c_R(0))}.
$$
By \cite[Theorem 4.12]{Adams}, $H^1(B^c_R(0))\hookrightarrow L^{2q'}(B^c_R(0))$ and there exists $C_{q'}>0$ (independent of $R$) such that
$$\|u\|_{L^{2q'}(B^c_R(0))}\le C_{q'}\|u\|_{H^1(B^c_R(0))},\,\,\forall u\in H^1(B^c_R(0)).$$
Thus
\be\lab{mpp1}
\int_{\mathbb R^2\setminus  B_R(0)}g_\epsilon(x, u_n)u_n\,\ud x\leq (C_1+CC_{q'}^2)\delta^2\le\delta/2,
\ee
by choosing $\delta$ small enough. On the other hand, since $g_\epsilon(x, u)u \in L^1(\mathbb R^2)$, we can choose $R>0$ large enough such that
\begin{align}\label{mpp2}
\int_{\mathbb R^2\setminus  B_R(0)}g_\epsilon(x, u)u\,\ud x\le\delta/2.
\end{align}
Combining \eqref{mpp1} and \eqref{mpp2}, the proof of Step 3 follows.
\vskip0.1in
\textbf{Step 4:} For every compact set $K\subset \mathbb R^2$
\[
\displaystyle\lim_{n\rightarrow\infty}\int_{K}g_\epsilon(x, u_n)\,\ud x=\displaystyle \int_{K}g_\epsilon(x, u)\,\ud x.
\]
Since $\{u_n\}$ is bounded in $H^1(\mathbb R^2)$. Hence $u_n\rightarrow u$ in $L^1(K)$ for any compact set $K\subset \mathbb R^2$ and $g_\epsilon(x, u_n)\in L^1(K)$. Moreover, since $\{u_n\}$ is a bounded $(PS)_c$ sequence, we have
\[
\sup_{n}\left|\ds \int_{\mathbb R^2}g_\epsilon(x, u_n)u_n\,\ud x\right|<\iy.
\]
Now the proof of the claim follows from \cite[Lemma 2.1]{MR1386960}.
\vskip0.1in
\textbf{Step 5:} For $R$ given above, we claim that
\[
\displaystyle\lim_{n\rightarrow\infty}\int_{B_R(0)}g_\epsilon(x, u_n)u_n\,\ud x=\displaystyle \int_{ B_R(0)}g_\epsilon(x, u)u\,\ud x.
\]
Since the sequence $\{u_n\}$ is bounded, up to a subsequence, $u_n\rightharpoonup u$ in $H^1(\mathbb R^2)$ and a.e. in $\mathbb R^2$ as $n\rg\iy$. Let $P(x,t)=g_\epsilon(x,t)t, Q(t)=e^{\alpha t^2}-1, t\in\R$, where $\alpha>4\pi$ with $\alpha\|\na u_n\|_2^2<4\pi$ for $n$ large. Obviously, by $(f_1)$-$(f_2)$,
$$
\lim_{t\rightarrow\infty}\frac{P(x,t)}{Q(t)}=0\,\,\,\mbox{uniformly for}\,\,\,x\in B_R(0)
$$
and
$$
\sup_n\int_{B_R(0)}Q(u_n)\,\ud x<\infty,\,\,\,P(x,u_n)\rightarrow P(x,u),\,\, a.e.\,\, x\in B_R(0).
$$
Strauss's compactness lemma yields the result desired.

\vskip0.1in
\textbf{Step 6:} We show that $u_n\rightarrow u$ in $H^1(\mathbb R^2)$ as $n\rg\iy$.\\
Since $\langle \tilde I_\epsilon(u_n), u_n\rangle\rightarrow 0$ as $n\rightarrow\iy$,
\begin{align}\label{dercon}
\ds \int_{\mathbb R^2}(|\nabla u_n|^2+V(\epsilon x)u_n^2)\,\ud x=\int_{\mathbb R^2}g_\epsilon(x, u_n) u_n\,\ud x+o_n(1).
\end{align}
By Step 4, it is easy to verify that $u$ is a critical point of $\tilde I_\epsilon$, so
\begin{align}\label{weakcon}
\ds \int_{\mathbb R^2}(|\nabla u|^2+V(\epsilon x)u^2)\,\ud x=\int_{\mathbb R^2}g_\epsilon(x, u) u\,\ud x.
\end{align}
Combing Step 3, Step 5, \eqref{dercon} and \eqref{weakcon}, we get
\[
\ds \lim_{n\rg\iy}\int_{\mathbb R^2}(|\nabla u_n|^2+V(\epsilon x)u_n^2)\,\ud x=\ds \int_{\mathbb R^2}(|\nabla u|^2+V(\epsilon x)u^2)\,\ud x.
\]
As a consequence, $u_n\rightarrow u$ in $H^1(\mathbb R^2)$ as $n\rg\iy$.
\end{proof}

\section{The limit problems}

As mentioned before, we study the following limit problem
\begin{equation}
\label{lmtp}  -\Delta u + k u = f(u)\,\,\mbox{in}\,\,\R^2
\end{equation}
for some $k >0$. The associated energy functional $\Phi_k : H^1(\RT) \to
\R$ is defined by
\begin{equation*}
\Phi_k(u)=\frac{1}{2} \int_{\RT} (|\nabla u|^2+k u^2)\,\ud x - \int_{\RT}
F(u)\,\ud x.
\end{equation*}
In \cite{MR2875652}, with the similar assumptions on $f(t)$ as in
Theorem \ref{thm1}, C. O. Alves, M. A. S. Souto and M. Montenegro proved that there exists a radially asymmetric ground state solution $U$ of \eqref{lmtp} (See also \cite{MR1846738} for similar results under more restrictive assumptions). Moreover, $U$ satisfies
\begin{equation}\label{eq}
\frac{1}{2}\int_{\mathbb R^2}|\nabla U|^2\,\ud x =\Phi_k(U)<\frac{1}{2}\ \mbox{and}\ \int_{\mathbb R^2}\left(F(U)-\frac{k}{2}U^2\right)\,\ud x=0.
\end{equation}
Moreover by \cite[Proposition 2.1]{MR3428453}, every radially asymmetric solution of \eqref{lmtp} is decreasing in $r=|x|$ and decays exponentially at infinity. Let $S_k$ be the set of nontrivial solutions of \eqref{lmtp} and $U\in H^1_{\mathrm{rad}}(\mathbb R^2)$ be a ground state solution of \eqref{lmtp}, then
\begin{equation*}
S^k_g:=\Phi_k(U)= \inf\{\Phi_k(u): u\in S_k\}.
\end{equation*}
 Now we define
\begin{equation} \label{mnmxlm}
m_k=\inf_{\gamma\in\Gamma_k}\max_{t\in[0,1]} \Phi_k(\gamma(t))
\end{equation}
with $\Gamma_k=\{ \gamma\in C([0,1],H^1_{\mathrm{rad}}(\mathbb R^2)): \ \gamma(0)=0, \Phi_k(\gamma(1))<0 \}$.
Following the proof of \cite[Corollary 1.5]{MR2875652}, $m_k=S^k_g$. Moreover, there exists $\gamma\in\Gamma_k$ such that $U\in\gamma ([0,1])$ and
\[
\max_{t\in[0,1]} \Phi_k(\gamma (t))=m_k.
\]
For $k=1$, for simplicity, let
$
 \Phi:= \Phi_1,\ m:= m_1.
$
Similar as in \cite[Lemma 2.4]{MR2900480}, we show the monotonicity of $m_k$ using the following compactness result in \cite{MR2875652}.
\begin{lemma} \label{smon}
Under the assumptions of Theorem \ref{thm1}, for any  sequence $\{v_n\}$  in $H_{\mathrm{rad}}^1(\mathbb R^2)$ such that
$
\sup_n\|\nabla v_n\|_{L^2(\mathbb R^2)}^2=\rho<1,\; \sup_n\|v_n\|_{L^2(\mathbb R^2)}^2<\infty
$
together with $v_n \rightharpoonup v$ in $H_{\mathrm{rad}}^1(\mathbb R^2)$ and a. e. in $\R^2$ as $n\rightarrow \infty$,
we have
$$
\lim_{n\rightarrow\infty}\int_{\mathbb R^2}F(v_n)\,\ud x=\int_{\R^2}F(v)\,\ud x.
$$
\end{lemma}
\begin{lemma}\label{monmk}
Under the assumptions of Theorem \ref{thm1}, we have
$$
0<m_k<\frac{1}{2}\left(\frac{1}{2}-\frac{1}{\mu}\right)
$$
and the map $m:(0,+\infty) \to (0,+\infty)$,
$
m(k) = m_k
$
is strictly increasing and continuous.
\end{lemma}
\begin{proof} Similar as in \cite{MR2875652}, we get that
$$
m_k\le\frac{p-2}{2p}\lambda_p^{-2/(p-2)}\mathcal{C}_p^{p/(p-2)}.
$$
Then, for $$\lambda_p>\left(\frac{p-2}{p}\cdot\frac{2\mu}{\mu-2}\right)^{\frac{p-2}{2}}\mathcal{C}_p^{\frac{p}{2}},$$ we have $m_k<\frac{1}{2}\left(\frac{1}{2}-\frac{1}{\mu}\right)$. Obviously, $m_k>0$. Now, we show the monotonicity of $m$. For any $0<k_1 < k_2$ and $ \gamma \in \Gamma_{k_2}$, it is clear that  $\gamma \in \Gamma_{k_1}$ and
$$ m_{k_1} \le \max_{t\in[0,1]} \Phi_{k_1}(\gamma(t)) < \max_{t\in[0,1]}
\Phi_{k_2}(\gamma(t))= m_{k_2}.$$
To prove the continuity of $m$, for any $\{k_j\}_j\subset(0,\iy)$ with $k_j\rg k>0$ as $j\rg\iy$, we take $\gamma \in \Gamma_k$, then for $j$
large enough $\gamma \in \Gamma_{k_j}$,
$$ m_{k_j} \le \max_{t\in[0,1]} \Phi_{k_j}(\gamma(t)) \to \max_{t\in[0,1]}
\Phi_{k}(\gamma(t))= m_{k}.$$
So
$
\limsup_{j\rg\iy} m_{k_j} \le m_k.
$
On the other hand, let $U_j$ be a radially symmetric ground state solution of \re{lmtp} with $k=k_j$, then $\{U_j\}_j$ is bounded in $H^1(\mathbb R^2)$ and $\|\na U_j\|_2^2=2m_{k_j}<1/2$. Up to a subsequence, we assume that $U_j \rightharpoonup U$ in $H_{\mathrm {\mathrm{rad}}}^1(\mathbb R^2)$ and a. e. in $\R^2$. Then by Lemma \ref{smon}
\[
\ds \int_{\mathbb R^2} F(U_j)\,\ud x \to \ds \int_{\mathbb R^2}  F(U)\,\ud x.
\]
Therefore, we conclude our result by the lower semi-continuity of the norm $\|\cdot\|$.

\end{proof}

\section{The min-max analysis} \label{The min-max argument}

This section is devoted to studying the min-max argument. Inspired by \cite{MR2900480}, let us define the following topological cone
\begin{equation*}
\mathcal{C}_\e=\left\{ \gamma_t(\cdot-\xi)\;\vline\; t\in[0,1], \xi\in
\overline{B_0^\e}\cap E \right\}
\end{equation*}
and a family of deformations of $\mathcal{C}_\e$:
\[
\Gamma_\e=\left\{\eta\in C\left(\mathcal{C}_\e,H^1(\RT)\right) \; \vline \;
\eta(u)=u, \ \forall  u \in \partial\mathcal{C}_\e \right\},
\]
where $\partial \mathcal{C}_\e$ is the topological boundary of $\mathcal{C}_\e$.
Here $\gamma_t=\gamma(t)$ is the curve at which the infimum in \eqref{mnmxlm} is achieved for $k=1$, see \cite[Lemma 2.1]{MR1974637}. Define the min-max level
\[
m_\e=\inf_{\eta\in\Gamma_\e}\max_{u\in\mathcal{C}_\e} \tilde{I}_\e(\eta(u)).
\]
Now following a similar idea as in \cite{MR2900480}, we state the following propositions without proof to make the paper self contained.
\begin{proposition}
\label{pr:bb} There exist $\e_0>0, \ \delta>0$ such that for
every $\e\in(0,\e_0)$
\begin{equation*} 
\tilde{I}_\e|_{\partial \mathcal{C}_\e} \leq m-\delta.
\end{equation*}
\end{proposition}

\begin{proposition}\label{estima}
$
\limsup\limits_{\e\to 0} m_{\e} \leq m.
$
\end{proposition}
Now we aim to give a similar estimate as in Proposition \ref{estima} from below. In order to do this we compare this level with another min-max level, which we will define with the help of barycenter type maps as below. Recall that in case of (V1), $E= \RT$, whereas, in case of (V2), $E$ is the space formed by eigenvectors associated to negative eigenvalues of $D^2V(0)$. Define $\pi_E$ as the orthogonal projection on $E$ and $h_\e : \RT \to E$ defined as $h_\e (x)=\pi_E(x)
\chi_{B_3^\e}(x)$, where $\chi_{B_3^\e}$ is the characteristic
function of $B_3^\e$. Let us define $\beta_\e: H^1(\RT)\setminus \{0\} \to E$ such that for any $u\in
H^1(\RT)\setminus \{0\}$
\[
\beta_\e (u) =\frac{\int_{\RT} h_\e(x) u^2\,\ud x}{\int_{\RT} u^2\,\ud x }.
\]
For $\delta >0$ small, let
$$
b_\e=\inf_{\Sigma \in \Xi_\e} \max_{u\in \Sigma}\tilde I_\e (u),
$$
where
\[
\Xi_\e=\left\{ \Sigma \subset H^1(\RT)\setminus\{0\} \left|
\begin{array}{l}
\Sigma \hbox{ is connected and compact,}
\\
\exists u_0,u_1 \in \Sigma \hbox{ s.t. }  \|u_0\|\leq \delta , \tilde
I_\e(u_1)<0,
\\
\forall u \in \Sigma,\  \beta_\e(u)=0.
\end{array}
\right. \right\}.
\]

Observe that, since $\tilde{I}_\e \geq \Phi_{\alpha_1}$, we have
$$ b_\e \geq m_{\alpha_1}>0.$$
Now similar to \cite[Lemma 3.3]{MR2900480}, we give the following lemma, whose proof is omitted.
\begin{lemma}
There exists $\e_0>0$ such that for any $\e \in (0 ,\e_0)$, \begin{equation*}
m_\e\geq b_\e.\end{equation*}
\end{lemma}

The following proposition gives us the desired estimate from below. We will prove this result in Section 6.
\begin{proposition}\label{fund}
$
\liminf\limits_{\e\to 0} m_{\e} \geq m.
$
\end{proposition}

Using the above propositions we get the following theorem on the existence of solutions for the truncated problem \eqref{truneq}.
\begin{theorem} There exists $\e_0>0$ such that for $\e \in
(0,\e_0)$ and $$\lambda_p>\left(\frac{p-2}{p}\cdot\frac{2\mu}{\mu-2}\right)^{\frac{p-2}{2}}\mathcal{C}_p^{\frac{p}{2}},$$
there exists a positive solution $u_{\e}$ of problem
\eqref{truneq}. Moreover, $\tilde{I}_{\e}(u_\e)=m_\e$.

\end{theorem}

\begin{proof}
By Propositions \ref{estima} and \ref{fund}, we deduce that $m_{\e} \to m$ as $\e \to 0$.  From Proposition \ref{pr:bb} we get that for small values of $\e$, $m_\e > \max_{\partial \mathcal{C}_\e} \tilde{I}_{\e}$. Moreover, from Proposition \ref{PSC} that $\tilde{I}_\e$ satisfies the $(PS)_c$ condition for $c<\frac{1}{2}\left(\frac{1}{2}-\frac{1}{\mu}\right)$. By Lemma \ref{monmk}, for $\e>0$ small, we know $m_\e<\frac{1}{2}\left(\frac{1}{2}-\frac{1}{\mu}\right)$. Therefore, a classical min-max theorem implies that $m_\e$ is a critical value of $\tilde{I}_\e$ for $\e$ sufficient small. Let us denote by $u_\e$ a critical point associated to $m_\e$. By the maximum principle, $u_\e$ is positive.
\end{proof}

\section{Proof of Proposition \ref{fund}}

In order to prove this proposition, we give some lemmas.

\begin{lemma}\label{constrp}
There exists $\e_0>0$ such that, for any $\e\in (0,\e_0)$,
there exist $u_\e\in H^1(\RT)$, with $\beta_\e(u_\e)=0$, and
$\lambda_\e\in E$ such that
\begin{equation}\label{constreq}
-\Delta u_\e+V(\e x) u_\e=g_\e( x, u_\e) +\lambda_\e \cdot
h_\e(x) u_\e
\end{equation}
and
$
\tilde I_\e (u_\e)=b_\e.
$
Moreover, $\{u_\e\}$ is bounded in $H^1(\RT)$.
\end{lemma}
\begin{proof}
Let $\e>0$ be fixed. By a classical min-max theorem, there exists a
sequence $\{u_n\} \subset H^1(\RT)$, which is a constrained (PS) sequence
at level $b_\e$ satisfying $\beta_\e(u_n)=0$ and $\{\lambda_n\}\subset E$ such
that
\begin{align}\label{psc1}
\tilde I_\e(u_n) \to b_\e, &\ \hbox{ as }n\to +\infty,\\\label{psc2}
\tilde I_\e'(u_n)-  \frac{\lambda_n \cdot h_\e(x)u_n}{\int_{\RT} u_n^2\,\ud x}
\to 0, &\ \hbox{ as }n\to +\infty.
\end{align}
In light of \eqref{psc1} and \eqref{psc2}, together with $\beta_\e(u_n)=0$ and
repeating the arguments of Proposition \ref{PSC}, we get that $\limsup_{n\rg\iy}\|\na u_n\|_2^2<\frac{2\mu c_0}{\mu-2}<\frac{1}{2}$ ($c_0$ is given in Proposition \ref{PSC}) and, therefore, up to a subsequence, it converges weakly to some $u_\e\in H^1(\RT)$. By choosing $R$ large enough such that $\phi_R h_\e=0$, it can be proved as in Proposition \ref{PSC} that $u_n\rightarrow u_\e$ in $H^1(\RT)$. Then $\tilde I_\e (u_\e)=b_\e.$ Since $b_\e>0$, we get $u_\e\not\equiv 0$ and $\beta_\e(u_\e)=\ds \lim_{n\rightarrow \infty}\beta_\e(u_n)=0$.
Next we claim that $\lambda_n$ is bounded. By $(f_1)$-$(f_2)$, for any given $\delta>0$ and $\gamma>4\pi$, there exists $C_\delta>0$ such that
$
|f(t)|\leq \delta |t|+C_\delta(e^{\gamma t^2}-1), \; t\in\R.
$
By H\"older's inequality and Lemma \ref{MTI}, for any $\phi\in C_0^\infty(\mathbb R^2)$, there exists $C_1,C_2>0$ such that
\begin{align*}
\left| \ds \int_{\RT} g_\e(x, u_n)\phi(x)\,\ud x\right|&\leq \ds\int_{\RT} \left|f(u_n)\phi(x) \right|\,\ud x\leq C_1\delta+C_\delta\left(\ds \int_{\mathbb R^2} (e^{q\gamma u_n^2}-1)\,\ud x\right)^\frac1q\le C_2,
\end{align*}
by choosing $\gamma>4\pi$ (close to $4\pi$) and $q>1$ ((close to $1$)) such that $q\gamma\|\na u_n\|_2^2<4\pi$ for $n$ large. Hence from \eqref{psc2},  to show boundedness of $\{\lambda_n\}$ it is enough to prove that
\[
{\ds \int_{\mathbb R^2}h_\varepsilon(x)u_n\phi\,\ud x}\;\textrm{is bounded away from zero}.
\]
For this
\begin{align*}
\ds \int_{\mathbb R^2}h_\varepsilon(x)u_n\phi\,\ud x=\ds \int_{B_3^\varepsilon}\pi_E(x)u_n\phi\,\ud x\rightarrow \ds \int_{B_3^\varepsilon}\pi_E(x)u_\varepsilon \phi\,\ud x.
\end{align*}
Since $u_\varepsilon \not \equiv 0$, we are done. Hence ${\lambda_n}$ is bounded in $E$ and converges strongly to some $\lambda_\e$. This completes the proof.
\end{proof}

\begin{lemma}\label{norm_control}
For $$\lambda_p>\left(\frac{p-2}{p}\cdot\frac{2\mu}{\mu-2}\right)^{\frac{p-2}{2}}\mathcal{C}_p^{\frac{p}{2}}$$ and $u_\e$ satisfying \eqref{constreq} with $\beta_\e(u_\e)=0$, we have $\limsup_{\e\rg0}\|\na u_\varepsilon \|_2^2<\frac{1}{2}$.
\end{lemma}
\begin{proof}
For $$\lambda_p>\left(\frac{p-2}{p}\cdot\frac{2\mu}{\mu-2}\right)^{\frac{p-2}{2}}\mathcal{C}_p^{\frac{p}{2}},$$ $m<\frac{1}{2}\left(\frac{1}{2}-\frac{1}{\mu}\right)$. By Proposition 3.2 and Lemma 3.1, $m_\e<\frac{1}{2}\left(\frac{1}{2}-\frac{1}{\mu}\right)\alpha_1$ for $\varepsilon$ small. Thus, similar as in Proposition 2.1, we have $\limsup_{\e\rg0}\|\na u_\varepsilon \|_2^2<\frac{1}{2}$.

\end{proof}

\bl\lab{l42}
$
\hbox{$u_\e\chi_{B_2^\e}\not\rg 0$ in $L^2(\R^2)$ as $\e\rg 0$.}
$
\el

\bp To the contrary, we assume that $u_\e\chi_{B_2^\e}\rg 0$ in
$L^2(\R^2)$ as $\e\rg 0$, then by the boundedness of $\{u_\e\}$ in
$H^1(\R^2)$ and the Sobolev interpolation inequality,
\be\lab{lq}
u_\e\chi_{B_2^\e}\rg 0\,\, \mbox{in}\,\, L^q(\R^2)\,\, \mbox{for any}\,\, q\ge2.
\ee
Recalling that $\beta(u_\e)=0$, we know $u_\e$ is nonnegative.
By the definition of $\ti{f}$, $|\ti{f}(u_\e)|\le au_\e$ which implies from \re{lq} that $\int_{B_2^\e\setminus B_1^\e}(1-\chi(\e x))\ti{f}(u_\e)u_\e\rg0$ as $\e\rg0$. Then
\begin{align*}
\int_{\R^2}(|\na u_\e|^2+V_\e(x)u_\e^2)\,\ud x&=\int_{B_2^\e}\chi_\e(x)f(u_\e)u_\e\,\ud x+\int_{\R^2\setminus B_2^\e}\ti{f}(u_\e)u_\e\,\ud x+o_\e(1)\\
&\le\int_{B_2^\e}\chi_\e(x)f(u_\e)u_\e\,\ud x+a\int_{\R^2\setminus B_2^\e}|u_\e|^2\,\ud x+o_\e(1),
\end{align*}
which implies that
\be\lab{yy}
\int_{\R^2}(|\na u_\e|^2+(V_\e(x)-a)u_\e^2)\,\ud x\le\int_{B_2^\e}\chi_\e(x)f(u_\e)u_\e\,\ud x+o_\e(1).
\ee
Now, we claim that $\lim_{\e\rg
0}\int_{B_2^\e}\chi_\e(x)f(u_\e)u_\e\,\ud x=0$. In fact, by $(f_1)$-$(f_2)$,  for any given $\gamma>4\pi$, there exists $C_\gamma>0$ such that
$
f(t)\leq C_\gamma t+t(e^{\gamma t^2}-1), \; \textrm{for any}\; t\ge0.
$
By H\"older's inequality,
\begin{align*}
\left| \ds \int_{B_2^\e}\chi_\e(x)f(u_\e)u_\e\,\ud x\right|&\leq \ds C_\gamma\int_{B_2^\e}u_\e^2\,\ud x+\left(\int_{B_2^\e}u_\e^{2q'}\,\ud x\right)^{\frac{1}{q'}}\left(\int_{B_2^\e}(e^{q\gamma u_\e^2}-1)\,\ud x\right)^{\frac{1}{q}},
\end{align*}
where $q,q'>1$ with $1/q+1/q'=1$, $\gamma>4\pi$ (close to $4\pi$) and $q>1$ (close to $1$) such that $q\gamma\|\na u_\e\|_2^2<4\pi$ for $\e$ small. Then the claim follows from Lemma \ref{MTI} and \re{lq}.

Finally, by \re{yy} and the claim above, we have $\|u_\e\|\rg0$ as $\e\rg0$, which contradicts the fact that $m_\e\ge b_\e\ge m_{\al_1}>0$. This concludes
the proof. \ep

\begin{lemma}\label{le:L4c}
$\|u_\e\|_{L^2((B_4^\e)^c)}\to 0$ as $\e\rightarrow 0$.
\end{lemma}
\begin{proof}
Let $\phi_\e :\RT\to \R$ be a smooth function such that
\[
\phi_\e(x)= \left\{
\begin{array}{ll}
0, & \hbox{in }B_3^\e,
\\
1, & \hbox{in }(B_4^\e)^c
\end{array}
\right.
\]
and with $0\leq  \phi_\e \leq 1$ and $|\nabla \phi_\e|\leq C\e$. By Lemma \ref{constrp} and $\phi_\e h_\e=0$, we have that
$
\langle I_\e'(u_\e),\phi_\e^2 u_\e\rangle=0.$
Since $(\chi_\e)$ and $(\phi_\e)$ have disjoint support, we get using the definition of $g_\e$,
\[
\int_{\RT} (|\nabla u_\e|^2+V(\e x)u_\e^2)\phi_\e^2\,\ud x +2\int_{\RT} u_\e \phi_\e \nabla
u_\e \cdot \nabla \phi_\e\,\ud x
\leq  \int_{\RT} a |u_\e \phi_\e|^2\,\ud x.
\]
Using H\"{o}lder's inequality, for some $C>0$,
\begin{equation}\label{tefm}
\int_{\RT}(|\nabla u_\e|^2+u_\e^2)\phi_\e^2\,\ud x \leq C\e.
\end{equation}
Observe that
\begin{align}\label{thodi}
\int_{\RT} |\nabla (u_\e \phi_\e)|^2\,\ud x&
\leq 2\int_{\RT} (|\nabla u_\e|^2 \phi_\e^2+u_\e^2|\nabla \phi_\e|^2)\,\ud x\leq 2\int_{\RT} |\nabla u_\e|^2 \phi_\e^2\,\ud x+C\varepsilon.
\end{align}
By the Sobolev embedding of $H^1(\mathbb R^2)\hookrightarrow L^2 (\mathbb R^2)$ and \eqref{thodi}-\eqref{tefm},
\begin{align*}
 a\int_{\RT} |u_\e \phi_\e|^2\,\ud x
 &\leq a\int_{\RT} (|\nabla (u_\e\phi_\e)|^2+ |u_\e \phi_\e|^2)\,\ud x\\
 &\leq 2a \int_{\RT} |\nabla u_\e|^2 \phi_\e^2\,\ud x+C\e+\int_{\RT} |u_\e|^2\phi_\e^2\,\ud x\\
 &\leq C\int_{\RT}(|\nabla u_\e|^2+u_\e^2)\phi_\e^2\,\ud x+C\e\rightarrow 0 \;\; \textrm{as}\;\; \e\rightarrow 0,
\end{align*}
which yields that $\|u_\e\|_{L^2((B_4^\e)^c)}\to 0$ as $\e\rightarrow 0$.
\end{proof}
\begin{lemma}\label{asylam}
$\lambda_\e=O(\e)$.
\end{lemma}
\begin{proof}
The proof follows the same steps as in \cite[Lemma 4.4]{MR2900480}.
\end{proof}
From Lemma \ref{asylam}, we can suppose that there exists $\bar \lambda\in E$ such
that
\[
\bar \lambda=\lim_{\e\to 0} \frac{\lambda_\e}{\e}.
\]
Let
\[
H_\e=\left\{x\in \RT \mid \bar \lambda \cdot x \leq
\frac{\alpha_1}{2\e}\right\}.
\]
\begin{lemma}\label{le:L2LpH}
$u_\e \chi_{H_\e} \nrightarrow 0$ in $L^2(\R^2)$ as $\e\rg 0$.
\end{lemma}

\begin{proof}
Let $H'_\e=\left\{x\in \RT \mid \bar \lambda  \cdot x\leq
\frac{\alpha_1}{3 \e}\right\}\subset H_\e$.  We show that
\begin{equation*}
\int_{H'_\e} u_\e^2\,\ud x \nrightarrow 0, \ \hbox{ as }\e\to 0.
\end{equation*}
Suppose by contradiction that
\begin{equation}\label{eq:absH}
\int_{H'_\e} u_\e^2\,\ud x \to 0, \ \hbox{ as }\e\to 0.
\end{equation}
Since $\beta_\e(u_\e)=0$ and $\bar \lambda\in E$, we have
\[
0= \int_{\RT} \bar{\lambda} \cdot h_\e(x) u_\e^2\,\ud x =\int_{(H'_\e)^c \cap
B_3^\e}\!\!\! \!\!\! \bar{\lambda} \cdot x \ u_\e^2 \,\ud x+\int_{H'_\e
\cap B_3^\e} \!\!\! \!\!\bar{\lambda} \cdot x \  u_\e^2\,\ud x \geq
\frac{\alpha_1}{3 \e}\int_{(H'_\e)^c \cap B_3^\e} \!\!\!\!\!\!
u_\e^2 \,\ud x+\int_{H'_\e \cap B_3^\e} \!\!\!\!\! \bar{\lambda} \cdot x \
u_\e^2\,\ud x.
\]
Therefore
\[
\frac{\alpha_1}{3 \e} \int_{(H'_\e)^c \cap B_3^\e}u_\e^2\,\ud x \leq
\left|\int_{H'_\e \cap B_3^\e} \bar{\lambda} \cdot x \
u_\e^2\,\ud x\right| \leq \frac{|\bar \lambda|R_3}{\e}\int_{H'_\e \cap
B_3^\e}u_\e^2\,\ud x
\]
and
\[
\int_{(H'_\e)^c \cap B_3^\e}u_\e^2\,\ud x \to 0, \ \hbox{ as }\e\to
0.
\]
This last formula, together with (\ref{eq:absH}), implies that
$u_\e \chi_{B_3^\e}\to 0$ in $L^2(\RT)$ but we get a
contradiction with Lemma \ref{l42} and so
the lemma is proved.
\end{proof}
Now we prove the compactness result of Lions type for the exponential critical growth. The idea is similar as in C. O. Alves, J. M. do \'O and O. H. Miyagaki  \cite{MR2036791}.

\begin{lemma} \label{lionsadm}
Let $(u_\varepsilon) \subset H^{1}(\mathbb{R}^{2})$ obtained in Lemma \ref{constrp}. If there exists $R>0$ such that
$$
\lim_{\varepsilon \to 0}\sup_{z \in \mathbb{R}^{2}}\int_{B_R(z)}|u_\varepsilon|^{2}\,\ud x=0,
$$
then
$$
\lim_{\varepsilon \to 0}\int_{\mathbb{R}^{2}}G_\varepsilon(x, u_\varepsilon)\,\ud x=\lim_{\varepsilon \to 0}\int_{\mathbb{R}^{2}}g_\varepsilon(x, u_\varepsilon)u_\varepsilon\,\ud x=0.
$$
\end{lemma}
\begin{proof}
Using the hypothesis
$$
\lim_{\varepsilon \to 0}\sup_{z \in \mathbb{R}^{2}}\int_{B_R(z)}|u_\varepsilon|^{2}\,\ud x=0,
$$
from \cite{MR850686}, we get that as $\e\rg0$,
\begin{equation}\label{consq}
u_\varepsilon\rightarrow 0\,\,\, \textrm{in}\; L^q(\mathbb R^2) \;\textrm{for all}\; q\in (2, \infty).
\end{equation}
By $(f_1)$-$(f_2)$, for every $\delta >0$ and $\gamma>4\pi$, there exists $C_\delta>0$ such that
$$
f(t)\leq \delta t+C_\delta t(e^{\gamma t^2}-1),\,\ t\ge0.
$$
Hence
\begin{align*}
\int_{\mathbb{R}^{2}}g_\varepsilon(x, u_\varepsilon)u_\varepsilon\,\ud x&\leq \int_{\mathbb{R}^{2}}f(u_\varepsilon)u_\varepsilon\,\ud x\leq \delta \int |u_\varepsilon|^2\,\ud x+C_\delta \int |u_\varepsilon|^2(e^{\gamma u_\varepsilon^2}-1)\,\ud x.
\end{align*}
By H\"older's inequality and $(e^t-1)^r\leq(e^{rt}-1)$ for $r>1$, we get
\begin{align*}
\int_{\mathbb{R}^{2}}g_\varepsilon(x, u_\varepsilon)u_\varepsilon\,\ud x&\leq \delta C+C_\delta  \left(\int_{\mathbb R^2}|u_\varepsilon|^{2q^\prime}\,\ud x\right)^\frac{1}{q^\prime}\left(\int (e^{\gamma q u_\varepsilon^2}-1)\,\ud x\right)^\frac{1}{q}.
\end{align*}
Choosing $\gamma>4\pi$ and $q>1$ with $q\gamma\|\na u_\varepsilon\|_2^2<4\pi$ for $\e$ small, by Lemma \ref{MTI}, we get
 \begin{align*}
\int_{\mathbb{R}^{2}}g_\varepsilon(x, u_\varepsilon)u_\varepsilon\,\ud x&\leq \delta  C+D_\delta  \left(\int_{\mathbb R^2}|u_\varepsilon|^{2q^\prime}\,\ud x\right)^\frac{1}{q^\prime}.
\end{align*}
Hence taking the limit as $\varepsilon\rightarrow 0$ together with \eqref{consq} and $(f3)$, we get the required result.
\end{proof}
The next result is a splitting Lemma for the exponential critical growth. In the higher dimension, for the Sobolev critical growth, a similar result has been obtained in \cite{MR3426106}. The idea of the proof is somewhat similar to \cite[Lemma 4.7]{MR3426106}.
\begin{lemma}\label{spliting}
Assume $(f1)$-$(f2)$. Let  $(u_k) \subset H^{1}(\mathbb{R}^{2})$ with $u_k\rightharpoonup u$ in $H^1(\mathbb R^2)$, a. e. in $\R^2$ and $\limsup_{k\rg\iy}\|\na u_k \|_2^2 < 1$. Then up to a subsequence,
\[
\ds \int_{\mathbb R^2}(f(u_k)-f(u)-f(u_k-u))\phi\,\ud x=o_k(1)\|\phi\|,\,\,\mbox{uniformly for}\,\,\phi\in C_0^\infty(\mathbb R^2).
\]
\end{lemma}
\begin{proof}
By $(f1)$, for any given $\varepsilon>0$, there exists a constant $c_\varepsilon\in (0, 1)$ such that $|f(t)|\leq \varepsilon |t|$ for all $|t|<2c_\varepsilon$. By $(f2)$, for every $\gamma>4\pi$ and $\varepsilon>0$, there exists $C_\varepsilon>2$ (large enough) such that $|f(t)|\leq \varepsilon (e^{\gamma t^2}-1)$ for every $|t|>C_\varepsilon-1$. Also using the continuity of the term $\frac{f(t)}{t}$ in a compact interval $[2c_\varepsilon, C_\varepsilon-1]$, we get
$|f(t)|\leq C|t|$ for $t \in[2c_\varepsilon, C_\varepsilon-1]$. Hence for all $t\in \mathbb R$, $|f(t)|\leq C(\varepsilon)|t|+\varepsilon (e^{\gamma t^2}-1)$. Now choose  $R=R(\varepsilon)$ such that
\be\lab{woy1}
\ds \int_{\mathbb R^2\setminus B(0, R)}|f(u)||\phi| \,\ud x\leq C\varepsilon \|\phi\|.
\ee
By the continuity of $f$, there exists $\delta \in (0, c_\varepsilon)$ such that  $|f(t_1)-f(t_2)|<\varepsilon c_\varepsilon$ for $|t_1-t_2|<\delta, |t_1|, |t_2|<C_\varepsilon+1$. Set $A_k:=\{x\in \mathbb R^2 \setminus B(0, R) | |u_k(x)|\leq c_\varepsilon\}$, by $|u_k-u|<c_\varepsilon+\delta<2c_\varepsilon$, we get
\begin{align*}
\ds \int_{A_k\cap\{|u|\leq \delta\}}|(f(u_k))-f(u_k-u)||\phi| \,\ud x&\leq \ds \varepsilon  \int_{A_k\cap\{|u|\leq\delta\}}(|u_k|+|u_k-u|)|\phi|\,\ud x\\&\leq \varepsilon (\|u_k\|_2+\|u_k-u\|_2)\|\phi\|_2\leq C\varepsilon\|\phi\|.
\end{align*}
Set $B_k:=\{x\in \mathbb R^2 \setminus B(0, R) | |u_k(x)|\geq C_\varepsilon\}$, then
\begin{align*}
\ds \int_{B_k\cap\{|u|\leq\delta\}}|(f(u_k))-f(u_k-u)||\phi| \,\ud x&\leq \ds \varepsilon  \int_{B_k\cap\{|u|\leq\delta\}}\left(\left(e^{\gamma u_k^2}-1\right)+\left(e^{\gamma (u_k-u)^2}-1\right) \right)|\phi|\,\ud x.
\end{align*}
Since $\|\na(u_k-u)\|_2^2=\|\na u_k\|_2^2-\|\na u\|_2^2+o_k(1)$, we get $\limsup_{k\rg\iy}\|\na(u_k-u)\|_2<1$. Hence
\begin{align*}
\ds \int_{B_k\cap\{|u|\leq\delta\}}|(f(u_k))-f(u_k-u)||\phi| \,\ud x&\leq \ds \varepsilon \|\phi\|\left(\int_{B_k\cap\{|u|\leq \delta\}}(e^{\gamma q u_k^2}-1)\,\ud x \right)^\frac1q\\
&+\varepsilon \|\phi\|\left(\int_{B_k\cap\{|u|\leq\delta\}}(e^{\gamma q (u_k-u)^2}-1)\,\ud x\right)^\frac1q.
\end{align*}
Now choosing $\gamma>4\pi$ (close to $4\pi$) and $q>1$ (close to 1) such that $\gamma q\|\na u_k\|_2^2<4\pi$ and $\gamma q\|\na (u_k-u)\|_2^2<4\pi$ for $k$ large, together with Lemma \ref{MTI}, for some $C>0$,
\[
\ds \int_{B_k\cap\{|u|\leq\delta\}}|(f(u_k))-f(u_k-u)||\phi| \,\ud x\leq \ds  C\varepsilon \|\phi\|,\,\,\,\mbox{for large}\,\, k.
\]
Set $C_k:=\{x\in \mathbb R^2 \setminus B(0, R) | c_\varepsilon\leq |u_k(x)|\leq C_\varepsilon\}$. Since $u_k\in H^1(\mathbb R^2)$, $|C_k|<\infty$. Then,
\begin{align*}
\ds \int_{C_k\cap\{|u|\leq\delta\}}|(f(u_k))-f(u_k-u)||\phi| \,\ud x&\leq \ds c_\varepsilon \varepsilon\|\phi\|_2|C_k|^\frac12\leq \varepsilon \|u_k\|_2\|\phi\|_2\leq C\varepsilon \|\phi\|.
\end{align*}
Thus from the above estimates,
\begin{equation}\label{outball}
\ds \int_{(\mathbb R^2\setminus B(0, R))\cap\{|u|\leq\delta\}}|(f(u_k))-f(u_k-u)||\phi| \,\ud x\leq C\varepsilon \|\phi\|,\,\,\,\mbox{for large}\,\, k.
\end{equation}
Note that,
$$
  |f(u_k)-f(u_k-u)|\leq C(\varepsilon)(|u_k|+|u_k-u|)+\varepsilon \left((e^{\gamma u_k^2}-1)+(e^{\gamma (u_k-u)^2}-1)\right).
$$
Hence
\begin{align*}
  &\ds \int_{(\mathbb R^2\setminus B(0, R))\cap\{|u|\geq\delta\}}|(f(u_k))-f(u_k-u)||\phi| \,\ud x\\&\leq\ds \varepsilon\int_{(\mathbb R^2\setminus B(0, R))\cap\{|u|\geq\delta\}}\left((e^{\gamma u_k^2}-1)+(e^{\gamma (u_k-u)^2}-1)\right)|\phi|\,\ud x\\&+C(\varepsilon)\ds \int_{(\mathbb R^2\setminus B(0, R))\cap\{|u|\geq\delta\}}(|u_k|+|u_k-u|)|\phi|\,\ud x.
  \end{align*}
  Now again arguing as before, $\limsup_{k\rg\iy}\|u_k-u\|<1$ for large $k$ and choosing $\gamma>4\pi$ (close to $4\pi$) and $q>1$ (close to $1$) such that $\gamma q\|\na u_k\|_2^2<4\pi$ and $\gamma q\|\na(u_k-u)\|_2^2<4\pi$ for $k$ large, we get
   \begin{align*}
  &\ds \int_{(\mathbb R^2\setminus B(0, R))\cap\{|u|\geq\delta\}}|(f(u_k))-f(u_k-u)||\phi| \,\ud x\\&\leq\ds \varepsilon\int_{(\mathbb R^2\setminus B(0, R))\cap\{|u|\geq\delta\}}\left((e^{\gamma u_k^2}-1)+(e^{\gamma (u_k-u)^2}-1)\right)|\phi|\,\ud x \leq C\varepsilon\|\phi\|.
\end{align*}
Since $u\in H^1(\mathbb R^2)$, $|{(\mathbb R^2\setminus B(0, R))\cap\{|u|\geq\delta\}}|\rightarrow 0$ as $R\rightarrow \infty$. Thus for large $R>0$, we get
  \begin{align*}
 & C(\varepsilon)\ds \int_{(\mathbb R^2\setminus B(0, R))\cap\{|u|\geq\delta\}}(|u_k|+|u_k-u|)|\phi| \,\ud x\\&\leq C(\varepsilon)(\|u_k\|_3+\|u_k-u\|_3)\|\phi\|_3|{(\mathbb R^2\setminus B(0, R))\cap\{|u|\geq\delta\}}|^\frac13
  \leq\varepsilon \|\phi\|.
  \end{align*}
  Hence
  \[
  \ds \int_{(\mathbb R^2\setminus B(0, R))\cap\{|u|\geq\delta\}}|(f(u_k))-f(u_k-u)||\phi| \,\ud x\leq C\varepsilon \|\phi\|,\,\,\,\mbox{for large}\,\, k.
  \]
Thus
\be\lab{woy2}
  \ds \int_{\mathbb R^2\setminus B(0, R)}|(f(u_k))-f(u_k-u)||\phi| \,\ud x\leq C\varepsilon \|\phi\|,\,\,\,\mbox{for large}\,\, k.
\ee
Finally, similar as above, for any $\e>0$, there exists $C_\e>0$ such that $|f(t)|\leq C_\varepsilon|t|+\varepsilon (e^{\gamma t^2}-1),t\in\R$, where $\gamma>4\pi$ (close to $4\pi$), $q>1$ and $\gamma q\|\na(u_k-u)\|_2^2<4\pi$ for $k$ large. Then
 \begin{align*}
  &\int_{B(0, R)}|f(u_k-u)||\phi| \,\ud x\\
  &\le C_\e\int_{B(0, R)}|u_k-u||\phi| \,\ud x+\e\int_{B(0, R)}(e^{\gamma (u_k-u)^2}-1)|\phi| \,\ud x\\
  &\le C_\e\left(\int_{B(0, R)}|u_k-u|^2\,\ud x\right)^{\frac{1}{2}}\left(\int_{B(0, R)}|\phi|^2\,\ud x\right)^{\frac{1}{2}}\\
  &\,\,\,\,\,\,+\e\left(\int_{B(0, R)}|\phi|^{q'}\,\ud x\right)^{\frac{1}{q'}}\left(\int_{B(0, R)}(e^{q\gamma (u_k-u)^2}-1)\,\ud x\right)^{\frac{1}{q}}.
 \end{align*}
By Lemma \ref{MTI} and $u_k\rightarrow u$ in $L^2(B(0, R))$, we know for $k$ large enough, there exists $C>0$ (independent of $\e$) such that
\[
   \ds \int_{B(0, R)}|f(u_k-u)||\phi| \,\ud x\leq C\varepsilon \|\phi\|.
  \]
Set $D_k:=\{x\in B(0, R): |u_k(x)-u(x)|>1\}$. Note that $u_k(x)\rightarrow u(x)$ a. e. in $B(0, R)$ as $k\rg\iy$. Hence $|D_k|\rightarrow 0$ as $k\rightarrow \infty$. Noting that $\|\na u\|_2\le\liminf_{k\rg\iy}\|\na u_k\|_2<1$,  similar as above, for some $C>0$ (independent of $\e$),
\begin{align*}
  \ds \int_{D_k}|f(u_k)-f(u)||\phi| \,\ud x&\leq C(\varepsilon)\ds \int_{D_k}(|u_k|+|u|)|\phi|\,\ud x+\varepsilon \ds \int_{D_k} \left((e^{\gamma u_k^2}-1)+(e^{\gamma u^2}-1)\right)|\phi|\,\ud x\\
  &\leq C(\varepsilon)\ds \int_{D_k}(|u_k|+|u|)|\phi|\,\ud x+C \varepsilon \|\phi\|.
  \end{align*}
Since $u_k\rightarrow u$ in $L^2(B(0, R))$, by using the fact that $|D_k|\rightarrow 0$ as $k\rightarrow \infty$, for some $\ti{C}>0$ (independent of $\e$),
$$
\int_{D_k}|f(u_k)-f(u)||\phi| \,\ud x\le \ti{C}\varepsilon \|\phi\|.
$$
  Now, since $u \in H^1(\mathbb R^2)$, $|u\geq L|\rightarrow 0$ as $L\rg\iy$. Hence choosing $L=L(\varepsilon)$ large enough such that
   \begin{align*}
  &\ds \int_{(B(0, R)\setminus D_k)\cap \{u\geq L\}}|f(u_k)-f(u)||\phi| \,\ud x\\
  &\leq C(\varepsilon)\ds \int_{(B(0, R)\setminus D_k)\cap \{u\geq L\}}(|u_k|+|u|)|\phi|\,\ud x
  +\varepsilon \ds \int_{(B(0, R)\setminus D_k)\cap \{u\geq L\}}\left((e^{\gamma u_k^2}-1)+(e^{\gamma u^2}-1)\right)|\phi|\,\ud x\\
  &\leq C(\varepsilon)(\|u_k\|_3+\|u\|_3)\|\phi\|_3|{(B(0, R)\setminus D_k)\cap\{|u|\geq L\}}|^\frac13+C \varepsilon \|\phi\|\leq C\varepsilon \|\phi\|.
  \end{align*}
  On the other hand, by the Lebesgue dominated convergence theorem,
  \[
  \ds \int_{(B(0, R)\setminus D_k)\cap\{u\leq L\}}|f(u_k)-f(u)||\phi| \,\ud x\leq C\varepsilon \|\phi\|,\,\,\,\mbox{for large}\,\, k.
  \]
Thus
   \begin{equation}\label{inball}
  \ds \int_{B(0, R)}|f(u_k)-f(u)-f(u_k-u)||\phi| \,\ud x\leq C\varepsilon \|\phi\|,\,\,\,\mbox{for large}\,\, k,
  \end{equation}
which, combing \re{woy1} and \re{woy2}, yields the result desired.
\end{proof}

\begin{proposition}\label{CCLion} There exist $n\in \mathbb N$, $\bar c>0$ and, for all $i=1,\ldots,n$,
there exist $y_\e^i\in B_2^\e \cap H_\e$, $\bar y_i \in B_2$
and $u_i\in H^1(\RT) \setminus \{0\}$ such that
\begin{align*}
&\e y_\e^i\to \bar y_i,\,\, |y_\e^i -y_\e^j|\to \infty, \ \hbox{ if }i\neq j ,
\\
&u_\e (\cdot +y_\e^i )\rightharpoonup u_i\ \hbox{in }H^1(\RT),\,\,\|u_i\|\geq \bar c,  
\\
&\| u_\e -\sum_{i=1}^n u_i (\cdot -y_\e^i )\|_{H^1(H_\e)} \to
0,
\end{align*}
and $u_i$ is a positive solution of
$
-\Delta u_i+V(\bar y_i) u_i=g(\bar y_i,u_i)+\bar \lambda\cdot \bar y_i
u_i.
$
\end{proposition}

\begin{proof}
Denote by $\tilde{u}_\e$ the even reflection of $u_\e|_{H_\e}$
with respect to $\partial H_{\e}$. Observe that $\{\tilde{u}_\e\}$
is bounded in $H^1(\RT)$ and does not converge to $0$ in $L^{2}(\RT)$
(recall Lemma \ref{le:L2LpH}). Now we claim that
\[
\ds \int_{\mathbb{R}^{2}}g_\e(x, u_\e)u_\e\,\ud x \nrightarrow 0\;\textrm{as}\; \e\rightarrow 0.
\]
Suppose not then using the fact the  $u_\e$ solves problem \eqref{constreq} we get that $u_\e\rightarrow 0$ in $H^1(\mathbb R^2)$ as $\e\rightarrow 0$ which contradicts the fact that $b_\e>0$. Noting that $\limsup_{\e\rg0}\|\na u_\varepsilon \|_2^2<\frac{1}{2}$, by Lemma \ref{norm_control}, and using the even symmetry of $\tilde u_\varepsilon$ we get $\limsup_{\e\rg0}\|\na \ti{u}_\varepsilon \|_2^2\le2\limsup_{\e\rg0}\|\na u_\varepsilon \|_2^2<1$. Noting that $0\leq \ds \int_{\mathbb{R}^{2}}g_\e(x, u_\varepsilon)u_\varepsilon\,\ud x\leq \ds\int_{\mathbb{R}^{2}}f(u_\varepsilon)u_\varepsilon\,\ud x \nrightarrow 0$ as $\varepsilon\rightarrow 0$, by Lemma \ref{lionsadm}, there exists $y^1_\e
\in \R^2$ with
\[
\int_{B(y^1_\e,1)}| u_\e|^2\geq c>0.
\]
By Lemma \ref{le:L4c},
we can assume that $y^1_\e \in H_\e\cap B^\e_5$. Then there
exists $u_1\in H^1(\RT)$ such that $v_\e^1=u_\e
(\cdot+ y_\e^1)\rightharpoonup u_1$, weakly in $H^1(\RT)$. Observe that $v_\e^1$ solves the equation
\begin{equation*}
-\Delta v_\e^1+V(\e x+\e y_\e^1) v_\e^1=g(\e x+ \e
y_\e^1,v_\e^1) +\lambda_\e\cdot  h_\e(x+y_\e^1) v_\e^1,
\end{equation*}
Now denote $\ds \lim_{\e\rightarrow 0}\e y_\e^1=\bar{y}_1$. By $(f_1)$-$(f_2)$, for every $\delta>0$ and $\gamma>4\pi$, there exists $C_\delta>0$ such that

\[
\left|\ds \int_{\mathbb R^2}g_\varepsilon(x+y_\varepsilon^1, v_\varepsilon^1(x))v_\varepsilon^1(x)\,\ud x\right|\leq  C\delta+C_\delta\left(\ds \int_{\mathbb R^2}(e^{\gamma q (v_\varepsilon^1)^2}-1)\,\ud x\right)^\frac1q,
\]
which is bounded for some $q, \gamma$ with $q \gamma\|\na v_\varepsilon^1\|_2^2=q \gamma\|\na u_\varepsilon\|_2^2<4\pi$.
By \cite[ Lemma 2.1]{MR1386960},
\[
\ds \lim_{\e\rightarrow 0}\int_{\RT} g(\e x+ \e
y_\e^1,v_\e^1) \phi \,\ud x=\int_{\RT} g(\bar{y}_1,u_1)\phi \,\ud x,\;\;\textrm{for all}\;\; \phi\in C_0^\infty(\mathbb R^2) \]
and so, passing to the limit,  $u_1$ is a nontrivial weak solution of
\[
-\Delta u_1+V(\bar y_1) u_1=g(\bar y_1,u_1)+\bar \lambda\cdot \bar y_1
u_1.
\]
Since $y_\e^1\in H_\e$, we have that $\bar \lambda\cdot \bar y_1\leq
\alpha_1/2$ and so  $\bar y_1\in B_2$ (otherwise $u_1$ should be
$0$) and, by $(f1)$, there exists $c>0$ such that
$ \|u_\e\| \geq \|u_1\|>c.$
Let us define $w_\e^1=u_\e - u_1(\cdot - y_{\e}^1)$. We consider
two possibilities: either $\| w_\e^1\|_{H^1(H_\e)} \to 0$ or not.
In the first case the proposition should be proved by taking $n=1$.
In the second case, there are still two sub-cases: either $\|
w_\e^1 \|_{L^{2}(H_\e)} \to 0$ or not.
\\
{\bf Step 1: } Assume that $\| w_\e^1 \|_{L^{2}(H_\e)}
\nrightarrow 0$.

In such case, we can repeat the previous argument to the sequence
$\{w_\e^1\}$: we take $\tilde{w}_\e^1$ as the even reflection of $w_\e^1$ with
respect to $\partial H_\e$. Observe that
$$
\|\na \tilde w_\varepsilon^1\|_2^2\le2\|\na w_\varepsilon^1\|_2^2=2\|\na u_\e(\cdot+y_{\e}^1) - \na u_1\|_2^2=2\|\na u_\e\|_2^2-2\|\na u_1\|_2^2+o_\e(1),
$$
which implies that $\limsup_{\e\rg0}\|\tilde \na w_\varepsilon^1\|_2<1.$
Applying Lemma \ref{lionsadm}, there exists $y^2_\e \in \mathbb R^2$ such that
\[
\int_{B(y^2_\e,1)}|w^1_\e|^2\,\ud x\geq c>0.
\]
By Lemma \ref{le:L4c}, we get the following
\[
\ds \int_{(B_4^\e)^c}|w^1_\e|^2\,\ud x\to 0\;\;\textrm{as}\;\;\e\rightarrow 0
\]
which implies $y_\e^2\in H_\e\cap B_5^\e$.
Therefore, as above, there exists $u_2\in H^1(\RT)\setminus\{0\}$ such
that $v_\e^2=w_\e^1 (\cdot+ y_\e^2)\rightharpoonup u_2$,
weakly in $H^1(\RT)$. Observe that $v_\e^2$ solves the following problem
\begin{align*}
-\Delta v_\e^2 &+V(\e x+\e y_\e^2) v_\e^2-g(\e x+ \e
y_\e^2,v_\e^2) -\lambda_\e\cdot  h_\e(x+y_\e^2) v_\e^2\\
&=(V(\bar{y}_1)-V(\e x+\e y_\e^2))u_1(y+y_\e^2-y_\e^1)+(\lambda_\e \cdot  h_\e(x+y_\e^2)
-\bar {\lambda}\cdot \bar {y}_1)u_1(y+y_\e^2-y_\e^1)\\
&+\left(g(\e x+ \e y_\e^2,u_\e(x+y_\e^2))-g(\e x+ \e y_\e^2,v_\e^2)-g(\bar y_1,u_1(y+y_\e^2-y_\e^1))\right).
\end{align*}
Take any $\phi\in C_0^\infty(\mathbb R^2)$ with $\|\phi\|=1$, then by the fact that $\ds \lim _{\e\rightarrow 0}\e y_\e^2=\bar y_2$, we get
\begin{align*}
&\ds \int_{\RT} (V(\bar{y}_1)-V(\e x+\e y_\e^2))u_1(y+y_\e^2-y_\e^1) \phi \,\ud x\leq \left(\ds \int_{\RT} |V(\bar{y}_1)-V(\e x+\e y_\e^2)|^2|u_1|^2 \,\ud x\right)^\frac12 \|\phi\|\\
&\rightarrow 0, \;\;\textrm{uniformly as}\;\;\e\rightarrow 0
\end{align*}
and
\begin{align*}
&\ds \int_{\RT} (\lambda_\e \cdot  h_\e(x+y_\e^2)
-\bar {\lambda}\cdot \bar {y}_1)u_1(y+y_\e^2-y_\e^1) \phi \,\ud x\leq \left(\ds \int_{\RT} |\lambda_\e \cdot  h_\e(x+y_\e^2)
-\bar {\lambda}\cdot \bar {y}_1|^2|u_1|^2 \,\ud x\right)^\frac12 \|\phi\| \\
&\rightarrow 0, \;\;\textrm{uniformly as}\;\;\e\rightarrow 0.
\end{align*}
Now
\begin{align}\label{gepl}
&\ds \int_{\RT} (g(\e x+ \e y_\e^2,u_\e(x+y_\e^2))-g(\e x+ \e y_\e^2,v_\e^2)-g(\bar y_1,u_1(y+y_\e^2-y_\e^1)))\phi(x+y_\e^1-y_\e^2)\,\ud x\\\nonumber
&=\ds \int_{\RT} (g(\e x+ \e y_\e^2,u_\e(x+y_\e^2))-g(\e x+ \e y_\e^2,v_\e^2)-g(\e x+\e y_\e^1, u_1(y+y_\e^2-y_\e^1)))\phi(x+y_\e^1-y_\e^2)\,\ud x\\\nonumber
&+\ds \int_{\RT} g(\e x+\e y_\e^1, u_1(y+y_\e^2-y_\e^1))-g(\bar y_1,u_1(y+y_\e^2-y_\e^1))\phi(x+y_\e^1-y_\e^2)\,\ud x\rightarrow 0,
\end{align}
uniformly as $\e\rightarrow 0$, by making use of Lemma \ref{spliting} in the first integral in \eqref{gepl} and the Lebesgue dominated convergence theorem in the second integral in \eqref{gepl}. Therefore
\[
-\Delta v_\e^2 +V(\e x+\e y_\e^2) v_\e^2-g(\e x+ \e
y_\e^2,v_\e^2) -\lambda_\e\cdot  h_\e(x+y_\e^2) v_\e^2=o_\e(1),
\]
where $o_\e(1)\rightarrow 0$ strongly in $H^{-1}(\mathbb R^2)$.
Next following \cite[Lemma 2.1]{MR1386960} again, we get
\[
\ds \lim_{\e\rightarrow 0}\int_{\RT} g(\e x+ \e
y_\e^2,v_\e^2) \phi \,\ud x=\int_{\RT} g(\bar{y}_2,u_2)\phi \,\ud x,\;\;\textrm{for all}\;\; \phi\in C_0^\infty(\mathbb R^2) \]
and passing to the limit ,  $u_2$ is a nontrivial weak solution of
\[
-\Delta u_2+V(\bar y_2) u_2=g(\bar y_2,u_2)+\bar \lambda\cdot \bar y_2
u_2.
\]
Also, by $(f1)$ and the weak convergence, there exists $c>0$ such that
$$ \|u_2\|>c, \|u_\e\|^2 \geq \|u_1\|^2+\|u_2\|^2+o_\e(1).$$
Next we show that , $|y_\e^1 -y_\e^2 | \to +\infty$ as $\e\rightarrow 0$. On the contrary, we assume $|y_\e^1 -y_\e^2 |<c'$ for $\e$ small and some $c'>0$. Since we have
\[
c\leq \ds\int_{B(y^2_\e,1)}|w^1_\e|^2\,\ud x=\int_{B(y^2_\e-y_\e^1,1)}|u_\e(x+y_\e^1)-u_1(x)|^2\,\ud x.
\]
let $\tilde v_\e(x)=u_\e(x+y_\e^1)-u_1(x)$, then $\tilde v_\e \rightharpoonup 0$ in $H^1(\mathbb R^2)$ and \[
\int_{B(0,1+c')}|\tilde v_\e|^2\,\ud x\geq c>0.
\]
Thus following the argument as above, we get a contradiction.

Let us define $w_\e^2:=w_\e^1 - u_2(\cdot - y_{\e}^2)=u_\e - u_1(\cdot - y_{\e}^1) - u_2(\cdot - y_{\e}^2)$. Again, if $\|w_\e^2\|_{H^1(H_\e)}\to 0$, the proof is completed for $n=2$. Suppose now that $\| w_\e^2\|_{H^1(H_\e)}\nrightarrow 0$, $\|
w_\e^2\|_{L^{p+1}(H_\e)} \nrightarrow 0$. In such case we can
repeat the argument again. Observe that we would finish in a finite number of steps to conclude the proof. The only possibility left is the following:
\begin{equation} \label{no} \mbox{ at a certain step } j,\ \ \| w_\e^j\|_{H^1(H_\e)}\nrightarrow 0, \mbox{ and } \|
w_\e^j\|_{L^{2}(H_\e)} \to 0,
\end{equation}
where $w_\e^j=u_\e - \sum_{k=1}^j u_k(\cdot - y_{\e}^k)$.
\\
\
\\
{\bf Step 2: } The assertion \eqref{no} does not hold.

Suppose by contradiction \eqref{no}, there exists $\delta>0$ such that
\begin{equation*}
\| u_\e \|^2_{H^1(H_\e)} \geq \sum_{k=1}^j  \| u_k(\cdot -
y_\e^k) \|^2_{H^1(H_\e)} +\delta.
\end{equation*}
Let us define
$$
H_\e^1 = \left\{x\in \RT \mid \bar \lambda  \cdot x \leq
\frac{a_2}{\e}\right\},
$$
where $\frac{\alpha_1}{2}< a_1<\frac{2\alpha_1}{3}$. We claim that
\begin{equation}\label{eq:vu1p}
\| w_\e^j \|_{L^{2}(H_\e^1)} \nrightarrow 0.
\end{equation}
Otherwise, let us fix $R>0$ large enough, similar as in \cite{MR3426106,MR2900480}, for $\e$ small enough we have
\begin{equation}\label{eq:palle}
\|u_\e\|_{H^1(H_\e\setminus\cup_{k=1}^jB(y_\e^k,2R))}^2\ge\frac{\dd}{2}.
\end{equation}
By $(f_1)$-$(f_2)$, for any $\sigma>0$ and $\gamma>4\pi$ (close to $4\pi$) with $\gamma\|\na u_\e\|_2^2<4\pi$ for $\e$ small, there exists $C_\sigma>0$ such that
$$
|g_\e(x,u_\e)u_\e|\le C_\sigma|u_\e|^2+\sigma(e^{\gamma u_\e^2}-1),\,\,x\in\R^2.
$$
By Lemma \ref{MTI} and choosing a cut-off function $\phi$ as in \cite{MR2900480}, multiplying \eqref{constreq} by $\phi u_\e$, we have
$$
\int_{H_\e \setminus \left ( \cup_{k=1}^j B(y_\e^j, 2R)\right
)}\left (|\nabla u_\e|^2 + V(\e x) u_\e^2 \right )\,\ud x - \frac C R-D\sigma \leq C_\sigma
\int_{H_\e^1 \setminus \left ( \cup_{k=1}^j B(y_\e^j, R)\right )}
u_\e^2 \,\ud x.
$$
for some $D>0$. Combining \re{eq:palle} and the arbitrariness of $D,\sigma$, we can get a contradiction.

\medskip Finally, by \re{eq:vu1p} repeating the procedure above and similar as in \cite{MR3426106,MR2900480}, we can finish the proof.

\end{proof}
\noindent{\bf Proof of Proposition \ref{fund}}
\bp
We distinguish two cases as in \cite{MR2900480}.

\textbf{Case 1: $\bar \lambda. \bar y_i\geq 0$} In this case, similar to \cite{MR3426106}, using $\beta_\e(u_\e)=0$ and Proposition \ref{CCLion}, we get $\bar \lambda.\bar y_i=0$ for all $i=1,2, \cdot, n$. Hence using $(f3)$
 \begin{align*}
 b_\e&=\tilde I_\e(u_\e)-\frac{1}{\mu}\langle I^\prime_\e(u_\e),u_\e\rangle=\ds \int_{\mathbb R^2}\left(\frac1\mu g_\e(x, u_\e)u_\e-G_\e(x, u_\e)\right) \,\ud x\\
 &\geq \ds \int_{B_2^\e\cap H_\e}\left(\frac1\mu g_\e(x, u_\e)u_\e-G_\e(x, u_\e)\right) \,\ud x\\
 &=\ds \sum_{i=1}^n \ds \int_{\mathbb R^2}\left(\frac1\mu g(\bar y_i, u_i)u_i-G(\bar y_i, u_i)\right) \,\ud x+o_\e(1)=\ds \sum_{i=1}^n J_{\bar y_i}(u_i)+o_\e(1),
 \end{align*}
  where
 $J_y : H^1(\mathbb R^2) \to \R$,
\begin{equation*}
J_y (u) = \frac{1}{2}\int_{\RT} (|\nabla u|^2 + V(y)u^2)\,\ud x -
\int_{\RT} G(y,u)\,\ud x.
\end{equation*}
Hence
\[
 \liminf_{\varepsilon \rightarrow 0}b_\varepsilon\geq \ds \sum_{i=1}^n J_{\bar y_i}(u_i).
 \]
Since $g(\bar y_i, s)\leq f(s)$ for all $s\geq 0$, we have
 \[
 \liminf_{\varepsilon \rightarrow 0}b_\varepsilon\geq \ds \sum_{i=1}^n m_{V(\bar y_i)}.
 \]
 In case of  $n\geq 2$ using the monotonicity of $m_k$, since
$\bar y_i\in B_2$, the conclusion holds. In case of $n=1$, by $\beta_\e (u_\e)=0$, we get that $\ds \int_{B_3^\e\cap H_\e}\pi _E(\e x)u_\e^2\,\ud x\rightarrow 0$ as $\e\rightarrow 0$. Hence, $\pi_E(\bar y_1)\ds \int_{\mathbb R^2}u_1^2\,\ud x=0$, which implies that $\bar y_1\in E^\perp$. Thus
$
 \liminf_{\varepsilon \rightarrow 0}b_\varepsilon\geq m_{V(\bar y_1)}>m.
$

\textbf{Case 2: $\bar \lambda. \bar y_i<0$} This case is similar to Case 2 in \cite{MR2900480}.
\ep
\section{Proof of Theorem \ref{thm1}}

In this section, we conclude the proof of Theorem \ref{thm1} by studying  the asymptotic behavior of the
solution obtained in Section \ref{The min-max argument}. Let $u_{\e}$ be a critical point of $\tilde{I}_\e$ at
level $m_{\e}$. The following result describes the behavior of $u_{\e}$ as $\e \to 0$. The proof is similar as in \cite{MR2900480}.
\begin{proposition} \label{hola} Given a sequence $\e_{j} \to 0$, there exist a
subsequence (still denoted by $\e_j$) and a sequence of points
$y_{\e_j} \in \R^N$ such that
\[
\e_j y_{\e_j} \to 0, \| u_{\e_j} - U(\cdot -
y_{\e_j})\| \to 0,
\]
where $U$ is a positive ground state solution of $-\Delta u+u=f(u)$.
\end{proposition}
To show that $u_\e$ solves problem \eqref{PE}, it is enough to prove that
$u_{\e}(x) \to 0$ as $\e \to 0$ uniformly in $x \in (B_1^{\e})^c$. By Proposition \ref{hola} and $U\in L^\infty(\mathbb R^2)$, we can use the idea from \cite{MR1846738,MR3428453} to get $\sup_{\e}\|u_\e\|_{L^\infty(\mathbb R^2)}<\iy$.
By Proposition \ref{hola}, we obtain
$$ \|u_{\e}\|_{H^1((B_0^{\e})^c)} \leq \|u_{\e} - U(\cdot - y_{\e})\| + \| U(\cdot -
y_{\e})\|_{H^1((B_0^{\e})^c)} \to 0,$$ as $\e \to 0$.
For any $x \in (B_1^{\e})^c$, $B(x, 2)\subset \mathbb R^2\setminus B_0^\e$. Thus by uniform elliptic estimates, there exists $C>0$, independent of $x$, such that
$$ u_{\e}(x) \leq C
\|u_{\e}\|_{H^1(B(x,2))}  \leq C \|u_{\e}\|_{H^1((B_0^{\e})^c)}\to
0$$
as $\e\rightarrow 0$.  This concludes the proof.
\qed

\end{document}